\documentclass{amsart}
\usepackage{amssymb}

\def\FF{{\mathbb F}}
\def\squareforqed{\hbox{\rlap{$\sqcap$}$\sqcup$}}
\def\qed{\ifmmode\squareforqed\else{\unskip\nobreak\hfil
\penalty50\hskip1em\null\nobreak\hfil\squareforqed
\parfillskip=0pt\finalhyphendemerits=0\endgraf}\fi\medskip}

\newcommand{\ombar}{\overline{\omega}}

\newcommand{\SL}{\mathrm{SL}}
\newcommand{\LL}{\mathrm{L}}
\newcommand{\PGL}{\mathrm{PGL}}
\newcommand{\PSL}{\mathrm{PSL}}

\newcommand{\U}{\mathrm{U}}

\newcommand{\udot}{{}^{\textstyle .}}
\newcommand{\diag}{\mathrm{diag}}
\newtheorem{theorem}{Theorem}
\newtheorem{proposition}{Proposition}
\newtheorem{lemma}{Lemma}

\newtheorem{remark}{Remark}

\title{On the compact real forms of the Lie algebras of type $E_6$ and $F_4$}
\date{
19th August 2012}

\author{Robert A. Wilson}
\address{School of Mathematical Sciences, Queen Mary University of London,
Mile End Road, London E1 4NS, UK}
\email{R.A.Wilson@qmul.ac.uk}
\commby{Alexander Kleshchev}

\begin{document}
\begin{abstract}
We give a construction of the compact real form of the Lie algebra of type $E_6$,
using the finite irreducible subgroup of shape
$3^{3+3}{:}\SL_3(3)$, which is isomorphic to a maximal subgroup of the
orthogonal group $\Omega_7(3)$. 
In particular we show that the algebra is uniquely determined by this subgroup.
Conversely, we prove from first principles that the algebra satisfies the
Jacobi identity, and thus give an elementary proof of existence of a Lie algebra
of type $E_6$.
The compact real form of $F_4$ is exhibited as a subalgebra.
\end{abstract}

\maketitle
\section{Introduction}
\label{intro}
The standard construction of the complex simple Lie algebras using their root systems
(see for example Carter's book \cite{Carter})
yields a so-called Chevalley basis, with respect to which the structure constants
of the algebra are all integers. This basis can therefore be used to define
Lie algebras over any field. In particular, these algebras over the real numbers 
are known as the {\em split real forms} of the simple Lie algebras. The
corresponding Lie groups are not compact.

On the other hand, Lie groups which arise in practical applications often are compact,
and it is desirable to have good constructions of these groups and the corresponding
algebras, both of which are known as the {\em compact real forms}. 
It is well-known that every complex simple Lie algebra has a unique
compact real form, and a suitable basis may be obtained from a Chevalley basis
by replacing each pair $\{e_r,e_{-r}\}$ of root vectors by the pair
$\{e_r+e_{-r}, \sqrt{-1}(e_r-e_{-r})\}$, and replacing  each basis vector
$h_r$ of the Cartan subalgebra by $\sqrt{-1}h_r$ (see for example p. 149 of
\cite{Jacobson}).

The simplest example of all is the algebra of type $A_1$. The split real form is
usually taken with respect to a basis $\{e,f,h\}$ and Lie products
\begin{eqnarray*}
[e,f]&=&h,\cr
 [h,e]&=&2e,\cr 
[h,f]&=&-2f.
\end{eqnarray*}
 Defining 
\begin{eqnarray*}
i&=&\sqrt{-1}h/2,\cr 
j&=&(e-f)/2,\cr
k&=&\sqrt{-1}(e+f)/2,
\end{eqnarray*} 
we easily compute 
$[i,j]=k$, $[j,k]=i$, and $[k,i]=j$, 
so that we obtain the familiar `cross product' on Euclidean $3$-space.
More generally, in the case of the
orthogonal groups, it makes sense to take an orthonormal basis $\{v_1,\ldots,v_n\}$
for the Euclidean space, on which the (compact) orthogonal group acts naturally.
Then the Lie algebra is essentially the exterior square of this module, so has a
basis $\{v_i\wedge v_j=-v_j\wedge v_i\}$, which is more or less equivalent to
the modified Chevalley basis described above. The Lie product is easily described
with respect to this basis by
$$[v_i\wedge j_j,v_j\wedge v_k]=v_i\wedge v_k$$ for distinct $i,j,k$, all other products
of basis vectors being $0$.

In the case of the five exceptional simple Lie algebras, $G_2$, $F_4$, $E_6$, $E_7$
and $E_8$, however, a more complicated change of basis may reveal some more
interesting structure.
There has been quite a lot of work on the compact real forms
by the Russian school (see for example the book by Kostrikin and Tiep \cite{KT}).
But even in this work, really nice constructions are hard to find.
In \cite{compactG2} I gave a construction of the compact real form of $G_2$ using
only the action of the group $2^3\udot \LL_3(2)$. In particular I showed that
this group determines the algebra. It acts by permuting $7$ mutually orthogonal
Cartan subalgebras, and the Lie multiplication is given by a single easy formula
and its images under the group.

Turning now to $E_6$,
it is well-known that the complex Lie group $E_6(\mathbb C)$ has a finite
subgroup of shape $3^{3+3}{:}\SL_3(3)$. Moreover, this subgroup is isomorphic
to the stabilizer of a maximal isotropic subspace (of dimension $3$) in the
finite simple orthgonal group $\Omega_7(3)$. Since it acts irreducibly on the
$78$-dimensional Lie algebra, it preserves a unique (up to scalars) Hermitian
form. Moreover, the representation is real, and if we write it as such
then the form becomes a quadratic form, which is (positive or negative)
definite. Therefore, this form is (again up to
scalars) the Killing form, and it follows that the Killing form is negative
definite, and the given $78$-dimensional real Lie algebra is the compact real
form of $E_6$.

In this paper, I construct this algebra from scratch using nothing more than
the structure of this finite group. In particular, the algebra is uniquely
determined (up to an overall scalar factor) by the group.
It may be hoped that this provides a useful way to calculate within the
compact real form of $E_6$. Moreover, the embedding of $F_4$ in $E_6$ is
reflected in the embedding of $3^3{:}\SL_3(3)$ in $3^{3+3}{:}\SL_3(3)$, and 
therefore we obtain also a simple description of the compact real form of $F_4$.
This is particularly revealing, as it is expressed in terms of a $13$-dimensional space
of quaternions, although of course the Lie product is not (bi-)linear
over quaternions.

There are a few related constructions in the literature, most notably that of
Burichenko \cite{Bur1} (see also Burichenko and Tiep \cite{BurTiep}). 
Our work overlaps with theirs, but goes a bit further: our formulae are a
little more concrete and explicit; we prove the existence of a Lie algebra
of type $E_6$ independently of the Chevalley construction; we generalise
to arbitrary fields of characteristic not $3$; and we express the subalgebra
of type $F_4$ in terms of the Hurwitz ring of integral quaternions.

There is also a very nice construction of the
triple cover $3\udot E_6(\mathbb C)$ in its $27$-dimensional representation
by Griess \cite{Griess}, using a Moufang loop of order $3^4$. This loop has
an automorphism group $3^3{:}\SL_3(3)$, and is analogous to the Moufang loop
of $2^4$ octonions $\{\pm1,\pm i_0,\ldots,\pm i_6\}$ which has automorphism 
group $2^3\udot \SL_3(2)$. Burichenko \cite{Bur2} has a similar
construction of the $27$-dimensional representation of $3\udot E_6(\mathbb C)$, 
and a related $27$-dimensional representation of $3\udot\Omega_7(3)$,
which are also briefly described in \cite[Section 14.1]{KT}.

\section{The group $3^{3+3}{:}\SL_3(3)$}
The easiest way to define the required group 
$3^{3+3}{:}\SL_3(3)$, as an abstract group, is to say that it is
isomorphic to the stabilizer in the simple orthogonal group $\Omega_7(3)$
of a maximal isotropic subspace (of dimension $3$) in the natural module.
However, I shall not be using this description here
(although it was used as input to some computer calculations which led to
the definitions below). Instead I shall proceed directly
to describing the action of this group on $78$-dimensional real Euclidean space.
The disadvantage of this approach, however, is that 
it is not easy to see
that our group has exactly the above structure, at least until a very late stage in the argument.

Recall that $\LL_3(3)$ (which can be thought of as any of $\SL_3(3)$, $\PSL_3(3)$ or
$\PGL_3(3)$, according to preference) 
is a group of automorphisms of the projective plane of order $3$.
This plane consists of $13$ points and $13$ lines, with each line consisting of
four points. The points may be labelled by the elements of the field $\FF_{13}$
of order $13$, in such a way that the lines are
$$\{t,t+1,t+3,t+9\}$$ for each $t\in \FF_{13}$. As a permutation group on these
$13$ points, $\LL_3(3)$ may be generated by the three permutations
\begin{eqnarray*}
a&:&t\mapsto t+1\cr
b&:&t\mapsto 3t\cr
c&=&(3,9)(4,X)(5,6)(7,E)
\end{eqnarray*}
where we write $X=10$, $E=11$, $T=12$ to avoid confusion later.

We take $13$ Euclidean spaces of dimension $6$, labelled $V_0$, \ldots, $V_{T}$
with subscripts in $\FF_{13}$ as before. 
Let $V$ be the orthogonal direct sum of the $V_t$.
Each $6$-space is written as a
$3$-dimensional complex space, 
with 
\begin{eqnarray*}
\omega&=&e^{2\pi i/3} = (-1+\sqrt{-3})/2,\cr
\theta&=&\sqrt{-3}=\omega-\ombar,
\end{eqnarray*}
and Euclidean norm equal to the usual Hermitian norm.
Then the $72$ roots of $E_6$
may be taken as the images, under 
coordinate permutations and multiples of each coordinate
by powers of $\omega$,
of 
\begin{eqnarray*}
\pm(\theta,0,0)&&\mbox{ (18 of these)},\cr
\pm(1,1,1)&&\mbox{ (54 of these)}.
\end{eqnarray*}
For any vector $v\in \mathbb C^3$, we write $v_t$ for the corresponding vector in $V_t$.

We are now ready to describe the actions of some elements on the $78$-space $V$.
First, the element $a$ of $\LL_3(3)$ lifts to an
element of order $13$ (also called $a$)
which maps each $v_t$ to $v_{t+1}$, so that for example 
$(\theta,0,0)_0\mapsto (\theta,0,0)_1$.
Second, the element $b$ maps $v_t$ to $v_{3t}$ and then multiplies by the
diagonal matrix $\diag(\omega,\ombar,\ombar)$, so that for example
$(1,1,1)_2\mapsto (\omega,\ombar,\ombar)_6$.
The action of the element $c$ is harder to describe: let $M_1$, $M_2$, $M_3$ and $M_4$
be the matrices
$$
\frac{\theta}{3}\begin{pmatrix}\omega&1&1\cr 1&\omega&1\cr 1&1&\omega\end{pmatrix},
\frac{\theta}{3}\begin{pmatrix}1&1&1\cr 1&\ombar&\omega\cr 1&\omega&\ombar\end{pmatrix},
\frac{\theta}{3}\begin{pmatrix}\omega&\omega&\omega\cr \ombar&\omega&1\cr \ombar&1&\omega\end{pmatrix},
\frac{\theta}{3}\begin{pmatrix}1&\omega&\omega\cr \ombar&\ombar&\omega\cr
\ombar&\omega&\ombar\end{pmatrix}
$$
respectively. Since the $M_i$ are unitary we have $M_i^{-1}=\overline{M}^\top$, so
$$\overline{M_i^{-1}}=M_i^\top.$$
Then $c$ is defined by
\begin{eqnarray*}
(x,y,z)_0&\mapsto&-(\overline{x},\overline{z},\overline{y})_0\cr
(x,y,z)_1&\mapsto&-({x},{z},{y})_1\cr
(x,y,z)_3&\leftrightarrow&-({x}\omega,{z}\ombar,{y}\ombar)_9\cr
(x,y,z)_4&\leftrightarrow&(\overline{x},\overline{y},\overline{z})_X\cr
(x,y,z)_T&\mapsto&(\overline{x},\overline{y},\overline{z})_T\cr
(x,y,z)_2&\mapsto&(\overline{x},\overline{y},\overline{z})_2 M_1\cr
(x,y,z)_8&\mapsto&(\overline{x},\overline{y},\overline{z})_8 M_2\cr
(x,y,z)_5&\leftrightarrow&(\overline{x},\overline{y},\overline{z})_6 M_3\cr
(x,y,z)_7&\leftrightarrow&(\overline{x},\overline{y},\overline{z})_E M_4
\end{eqnarray*}
For clarity we add also:
\begin{eqnarray*}
(x,y,z)_9&\mapsto& -(x\ombar, z\omega,y\omega)_3\cr
(x,y,z)_X&\mapsto& (\overline{x},\overline{y},\overline{z})_4\cr
(x,y,z)_6&\mapsto&  (\overline{x},\overline{y},\overline{z})_5M_3^\top\cr
(x,y,z)_E&\mapsto& (\overline{x},\overline{y},\overline{z})_7M_4^\top.
\end{eqnarray*}

It is clear that the permutation action of the group given by
these generators on the $13$ subspaces $V_0$, \ldots, $V_T$, is exactly the
standard permutation action of $\SL_3(3)$.
In fact, the kernel of this action is trivial,
so that $a$, $b$ and $c$ generate a group isomorphic to $\SL_3(3)$, but we
shall not need this fact, and we shall not prove it here.
(It is in any case straightforward to show that all the generators
preserve the set of $13\times 72=936$ roots of the $13$ copies of 
$E_6$, after which the order of the group can be easily obtained computationally.)

Next we need to specify some elements generating the normal subgroup $3^{3+3}$.
First, the normal subgroup of order $3^3$ is generated by conjugates of
an element $d$ acting as powers of $\omega$ on each $6$-space, as follows:
$$(1,\omega,1,\omega,\omega,\omega,\ombar,\ombar,1,\omega,\ombar,\omega,1).$$
That is, $v_0\mapsto v_0$, $v_1\mapsto \omega v_1$, and so on.
\begin{lemma}
\label{lemmaD}
The group $\langle D,a\rangle$ is of shape $3^3{:}13$, in which the normal subgroup
of shape $3^3$ is generated by $d,d^a,d^{a^2}$.
\end{lemma}
\begin{proof}
 First note that under pointwise multiplication the product of
$$(1,\omega,1,\omega,\omega,\omega,\ombar,\ombar,1,\omega,\ombar,\omega,1)$$
with
$$(\omega,1,\omega,\omega,\omega,\ombar,\ombar,1,\omega,\ombar,\omega,1,1)$$
is
$$(\omega,\omega,\omega,\ombar,\ombar,1,\omega,\ombar,\omega,1,1,\omega,1).$$
Thus  $$dd^{a^{-1}}=d^{a^{-3}},$$
which implies that  the minimum polynomial of $a$ in its action on the conjugates of $d$ is $x^3+x^2-1$.
In particular, the conjugates of $d$ by powers of $a$
generate an elementary abelian group of order
$3^3$.
\end{proof}
Modulo this group $D$, the next $3^3$-factor acts monomially on each $6$-space, 
generated by conjugates of an element $e$ which acts as follows:
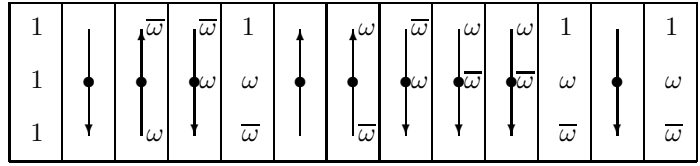
\begin{center}
\begin{picture}(260,70)
\multiput(0,0)(0,60){2}{\line(1,0){260}}
\multiput(0,0)(20,0){14}{\line(0,1){60}}
\put(10,10){
\multiput(20,20)(20,0){3}{\circle*{4}}
\multiput(100,20)(20,0){5}{\circle*{4}}
\put(220,20){\circle*{4}}
\multiput(20,40)(40,0){2}{\vector(0,-1){40}}

\multiput(140,40)(20,0){3}{\vector(0,-1){40}}
\put(220,40){\vector(0,-1){40}}

\multiput(100,0)(20,0){2}{\vector(0,1){40}}
\put(40,0){\vector(0,1){40}}
}
\put(8,8){
\put(0,0){$1$}
\put(0,20){$1$}
\put(80,0){$\ombar$}\put(80,20){$\omega$}
\put(200,0){$\ombar$}\put(200,20){$\omega$}
\put(80,40){$1$}\put(200,40){$1$}
\put(0,40){$1$}
\put(240,0){$\ombar$}\put(240,20){$\omega$}\put(240,40){$1$}
}
\put(12,8){
\put(60,40){$\ombar$}\put(60,20){$\omega$}
\put(40,40){$\ombar$}\put(40,0){$\omega$}
\put(180,40){$\omega$}\put(180,20){$\ombar$}
\put(120,0){$\ombar$}\put(120,40){$\omega$}
\put(140,20){$\omega$}\put(140,40){$\ombar$}
\put(160,40){$\omega$}\put(160,20){$\ombar$}
}

\end{picture}
\end{center}

\begin{lemma}
The group $H = \langle a,b,d,e\rangle$ has the shape
$3^{3+3}{:}13{:}3$, in which the normal subgroup $E\cong 3^{3+3}$ is generated by
$e,e^a,e^{a^2}$.
\end{lemma}
\begin{proof}
A similar calculation applied to $e$ yields
$$e^a.e=e^{a^3}.d^{-1},$$ so that modulo the group
$D=\langle d,d^a,d^{a^2}\rangle\cong3^3$, 
the conjugates of $e$ generate another $3^3$. 
This time the minimum polynomial of the action of $a$ is $x^3-x-1$.
Therefore the group generated by $a$, $d$
and $e$ has the shape $3^{3+3}{:}13$. It remains to check that $b$ normalizes this group.
In fact it is easy to see that $a^b=a^3$, and not much harder to check that $e^b=e$,
so this completes the proof.
\end{proof}

Let $G$ be the group (in fact of shape $3^{3+3}{:}\LL_3(3)$)
generated by $a,b,c,d,e$, and let  
$L$
denote the group $\langle a, b, c\rangle$, which is in fact isomorphic to $\SL_3(3)$
(although we have not proved this here). The element $c$ fixes the points $0,1,2,8,T$,
and therefore the subgroup $$F=\langle E, c,c^{a^{-1}},c^{a^{-2}},c^{a^5}, c^a\rangle$$ 
fixes the point $0$. In the action of $F$ on $V_0$, we see that $d,d^a,e$ lie in
the kernel, and $e^a,e^{a^2}$ generate an extraspecial group of order $3^3$,
which is the image of $E$ so is obviously normal in $F$ modulo the kernel of the action.
A little calculation shows that $F$ acts on $V_0$ as
$3^{1+2}{:}2S_4$, which is sometimes known as $\Gamma\U_3(2)$, and is isomorphic
to a maximal subgroup of the Weyl group of type $E_6$, which is itself isomorphic
to $\Sigma\U_4(2)$.
It is easy to see that $F$ acts irreducibly on $V_0$, and therefore $G$ acts
irreducibly on $V$.

Similarly, the five lines fixed by $c$ are 
$$\{0,1,3,9\},\{1,2,4,X\}, \{5,6,8,1\}, \{E,T,1,7\}, \{T,0,2,8\},$$ 
so that the line $\{0,1,3,9\}$
is fixed by $$c,c^{a^{-1}},c^{a^{-5}},c^{a^2},c^a.$$ 
These elements act on the line as the permutations
$(3,9)$, $(3,9)$, $(0,1)$, $(0,9)$, and $1$ respectively, so induce
the full $S_4$ of permutations. Indeed, they generate the full line stabilizer
$3^2{:}2S_4$ inside $\LL_3(3)$.
The elements of $D$ which
act non-trivially 
on this line act as 
$$(1,\omega,\omega,\omega), (\omega,1,\omega,\ombar),
(\omega,\omega,\ombar,1), (\omega,\ombar,1,\omega)$$ or their inverses.
These elements will be used frequently in the sequel.

\section{The Lie product}
In this section we show that there is (up to real scalar multiplication)
a unique bilinear product invariant under the action of the group
$G=\langle a,b,c,d,e\rangle$ on the $78$-space $V$,
and moreover that this product satisfies the Jacobi identity. 
First we show that there is at most
one such product, and then we show that any product so defined is anti-symmetric,
before using this to prove that there  is indeed a non-zero such product
invariant under $G$, and that this product satisfies the Jacobi identity.
\begin{lemma}\label{unique}
Up to a real scalar multiplication, there is a unique $G$-invariant bilinear product on $V$.
\end{lemma}
\begin{proof}
Since $G$ acts $2$-transitively on the $13$ spaces $V_t$, it is enough to determine
the product on $V_0\times V_0$ and on $V_0\times V_1$. Now the action of the conjugates
of $d$ shows immediately that the product is zero on $V_0\times V_0$, and that the
product of any vector in $V_0$ with any vector in $V_1$ lies in $V_3+V_9$.

Since the group $3^{3+3}$ acts irreducibly on $V_0$, we only need to consider products
of $(1,0,0)_0$ with $V_1$. Moreover, the stabilizer in $3^{3+3}$ of $(1,0,0)_0$ permutes
the nine spanning vectors 
$$(\omega^i,0,0)_1, (0,\omega^i,0)_1, (0,0,\omega^i)_1$$ 
of $V_1$ transitively,
so it is sufficient to determine the product of $(1,0,0)_0$ with $(1,0,0)_1$.

The element $e^a$ fixes $(1,0,0)_0$ and $(1,0,0)_1$, and maps
\begin{eqnarray*}
(x,y,z)_3&\mapsto& (y\ombar,z,x\omega)_3,\cr
(x,y,z)_9&\mapsto& (z\omega,x\ombar,y)_9.
\end{eqnarray*}
 Therefore the product of $(1,0,0)_0$
with $(1,0,0)_1$ lies in $\mathbb C(\ombar,1,1)_3 + \mathbb C(\omega,1,1)_9$.

Next, $c^a$ fixes $(1,0,0)_0$ and negates $(1,0,0)_1$, and acts on
$\mathbb C(\ombar,1,1)_3$ by fixing $\theta(\ombar,1,1)_3$ and negating $(\ombar,1,1)_3$.
Similarly, it fixes $\theta(1,\ombar,\ombar)_9$ and negates $(1,\ombar,\ombar)_9$.
Therefore the given product lies in $\mathbb R(\ombar,1,1)_3 + \mathbb R(1,\ombar,\ombar)_9$.

Finally, $c$ itself acts by negating both $(1,0,0)_0$ and $(1,0,0)_1$,
and interchanging $(\ombar,1,1)_3$ with $-(1,\ombar,\ombar)_9$. Therefore the given product
is a real multiple of $$(\ombar,1,1)_3-(1,\ombar,\ombar)_9.$$
Hence there is up to scalars 
at most one bilinear product invariant under $G$, as claimed.
\end{proof}

Our strategy for showing that such a (non-zero)
product actually exists divides into four steps, which are dealt with in Lemmas
\ref{step1}, \ref{step2}, \ref{step3} and \ref{step4} respectively:
\begin{enumerate}
\item show that any such product is anti-symmetric;
\item find two particular values of the product whose images under $H$
are sufficient to define the whole product;
\item prove that this product is well-defined, that is, it is invariant under $H$;
\item prove that this product is invariant under $c$.
\end{enumerate}

\begin{lemma}\label{step1}
Any $G$-invariant bilinear product on $V$ is anti-symmetric.
\end{lemma}
\begin{proof}
From the proof of Lemma~\ref{unique}, we may assume that
the product, written $[u,v]$, satisfies 
$$[(1,0,0)_0,(1,0,0)_1]=(\ombar,1,1)_3-(1,\ombar,\ombar)_9.$$
Applying $c^{a^{-5}}$ to this equation gives
\begin{eqnarray*}
[\frac{\theta}{3}(\omega,\omega,\omega)_1,\frac{\theta}{3}(\omega,\ombar,\ombar)_0]
&= &(\omega,1,1)_3 M_2 + (1,\ombar,\ombar)_9\cr
&=& (\omega,\ombar,\ombar)_3+(1,\ombar,\ombar)_9
\end{eqnarray*}
We now calculate the product of these two vectors the other way round.
Applying $e$ to the defining equation gives 
$$[(1,0,0)_0,(1,0,0)_1]=[(1,0,0)_0,(0,1,0)_1]=[(1,0,0)_0,(0,0,1)_1].$$
Similarly, applying other conjugates of $e$ and $d$ we obtain the following
multiplication table:
$$
\begin{array}{c|ccc|}
&(1,0,0)_1&(0,1,0)_1&(0,0,1)_1\cr\hline
(1,0,0)_0 &(\ombar,1,1)_3-(1,\ombar,\ombar)_9&
(\ombar,1,1)_3-(1,\ombar,\ombar)_9&(\ombar,1,1)_3-(1,\ombar,\ombar)_9\cr
(0,1,0)_0&(1,\ombar,1)_3-(\omega,\omega,\ombar)_9&
(\ombar,\omega,\ombar)_3-(1,1,\omega)_9&
(\omega,1,\omega)_3-(\ombar,\ombar,1)_9\cr
(0,0,1)_0&(1,1,\ombar)_3-(\omega,\ombar,\omega)_9&
(\omega,\omega,1)_3-(\ombar,1,\ombar)_9&
(\ombar,\ombar,\omega)_3-(1,\omega,1)_9\cr\hline
\end{array}
$$
In fact, applying conjugates of $d$ is quite easy:
if we multiply $v_0$ by $\omega$ and fix $v_1$ then we must
multiply $v_3$ by $\omega$ and $v_9$ by $\ombar$. On the other hand,
if we fix $v_0$ and multiply $v_1$ by $\omega$, then we must multiply
both $v_3$ and $v_9$ by $\omega$.
This leads quickly to the equations
\begin{eqnarray*}
[(\omega,0,0)_0,(\omega,\omega,\omega)_1]&=& 3(\omega,\ombar,\ombar)_3 -
3(1,\ombar,\ombar)_9\cr
[(0,\ombar,0)_0,(\omega,\omega,\omega)_1]&=&0\cr
[(0,0,\ombar)_0,(\omega,\omega,\omega)_1]&=&0
\end{eqnarray*}
from which we obtain
\begin{eqnarray*}
[\frac{\theta}{3}(\omega,\ombar,\ombar)_0,\frac{\theta}{3}(\omega,\omega,\omega)_1]
&=& -(\omega,\ombar,\ombar)_3 - (1,\ombar,\ombar)_9\cr
&=& -
[\frac{\theta}{3}(\omega,\omega,\omega)_1,\frac{\theta}{3}(\omega,\ombar,\ombar)_0]
\end{eqnarray*}
Since, by Lemma~\ref{unique}, 
this single non-zero value of the product defines the whole multiplication, it follows that
the whole multiplication is anti-symmetric. 
\end{proof}

\begin{lemma}\label{step2}
Any $G$-invariant product on $V$ is determined by anti-symmetry and the 
images under $H$ of
just two products, which may be taken (up to an overall scalar multiplication) to be
\begin{eqnarray*}
[(1,0,0)_0,(1,0,0)_1]&=&(\ombar,1,1)_3-(1,\ombar,\ombar)_9,\cr
[(1,0,0)_1,(1,0,0)_9]&=& -(1,\omega,\omega)_0+(\ombar,\omega,\omega)_3.
\end{eqnarray*}
\end{lemma}
\begin{proof}
Since the group $13{:}3$ generated by $a$ and $b$ has just two (regular) orbits on
unordered pairs of the $13$ points, represented by $\{0,1\}$ and $\{1,9\}$,
it suffices to determine the product on $V_1\times V_9$.
A similar argument to that given in the second paragraph of
the proof of Lemma~\ref{unique}
shows that it is sufficient to determine
$[(1,0,0)_1,(1,0,0)_9]$.

Applying $c$ to the equation
$$[(1,0,0)_8,(1,0,0)_9]=(\ombar,1,1)_E-(1,\ombar,\ombar)_4$$
gives
\begin{eqnarray*}
-[\frac{\theta}{3}(1,1,1)_8,(\ombar,0,0)_3] &=&
(\omega,1,1)_7 M_4^\top - (1,\omega,\omega)_X\cr
&=& (\theta\omega,0,0)_7 - (1,\omega,\omega)_X\cr
\Rightarrow
[\frac{\theta}{3}(1,1,1)_1,(\ombar,0,0)_9]&=&
-(\theta\omega,0,0)_0 + (1,\omega,\omega)_3\cr
\Rightarrow
[\frac{\theta}{3}(1,1,1)_1,(1,0,0)_9]&=&
-(\theta,0,0)_0+(\ombar,1,1)_3
\end{eqnarray*}
Now apply $e^{a^{-3}}$ to obtain
\begin{eqnarray*}
[\frac{\theta}{3}(1,\omega,\ombar)_1,(1,0,0)_9]&=&
-(0,\theta\omega,0)_0+(\omega,1,\omega)_3\cr
[\frac{\theta}{3}(1,\ombar,\omega)_1,(1,0,0)_9]&=&
-(0,0,\theta\omega)_0+(\omega,\omega,1)_3
\end{eqnarray*}
and add up the last three equations to get
\begin{eqnarray*}
[(\theta,0,0)_1,(1,0,0)_9]&=& 
-\theta(1,\omega,\omega)_0-\theta(\ombar,\omega,\omega)_3\cr
\Rightarrow
[(1,0,0)_1,(1,0,0)_9]&=& -(1,\omega,\omega)_0+(\ombar,\omega,\omega)_3
\end{eqnarray*}
as required.
To obtain all values of the product we now only need to apply elements of $H$
to the two values given, and use anti-symmetry and bilinearity.
\end{proof}

\begin{lemma}\label{step3}
There is a unique $H$-invariant anti-symmetric product on $V$
with 
\begin{eqnarray*}
[(1,0,0)_0,(1,0,0)_1]&=&(\ombar,1,1)_3-(1,\ombar,\ombar)_9,\cr
[(1,0,0)_1,(1,0,0)_9]&=& -(1,\omega,\omega)_0+(\ombar,\omega,\omega)_3.
\end{eqnarray*}
\end{lemma}
\begin{proof}
Most of the group $H\cong 3^{3+3}{:}13{:}3$ is used to obtain new values for the product
from the two given. Indeed, the conjugates of $d$ give the multiples by $\omega$,
and two conjugates of $e$ permute the three coordinates in the two factors.
It remains to check that a conjugate of $e$ which fixes the two factors
also fixes the product. In the case of $[(1,0,0)_0,(1,0,0)_1]$ this fact was
used to determine the product in the first place, so has already been checked.
The other case is similarly easy.
Finally, the quotient $13{:}3$ acts regularly on two orbits each of $39$
unordered pairs, giving the full product on $L$.
\end{proof}

To assist with computations, we provide a fuller version of the multiplication table
in Table~\ref{E6mult}. This must be used in combination with the action of $D$,
which shows how to compute the products of vectors with coordinates
$\omega$ or $\ombar$.

\begin{table}
$$\begin{array}{c|ccc|}

&(1,0,0)_1&(0,1,0)_1&(0,0,1)_1\cr\hline
(1,0,0)_0 &(\ombar,1,1)_3-(1,\ombar,\ombar)_9&
(\ombar,1,1)_3-(1,\ombar,\ombar)_9&(\ombar,1,1)_3-(1,\ombar,\ombar)_9\cr
(0,1,0)_0&(1,\ombar,1)_3-(\omega,\omega,\ombar)_9&
(\ombar,\omega,\ombar)_3-(1,1,\omega)_9&
(\omega,1,\omega)_3-(\ombar,\ombar,1)_9\cr
(0,0,1)_0&(1,1,\ombar)_3-(\omega,\ombar,\omega)_9&
(\omega,\omega,1)_3-(\ombar,1,\ombar)_9&
(\ombar,\ombar,\omega)_3-(1,\omega,1)_9\cr\hline
\cr

&(1,0,0)_3&(0,1,0)_3&(0,0,1)_3\cr\hline
(1,0,0)_0 &(\omega,1,1)_9-(\omega,\omega,\omega)_1&
(1,\ombar,\ombar)_9-(1,1,1)_1&(1,\ombar,\ombar)_9-(1,1,1)_1\cr
(0,1,0)_0&(\omega,\omega,\ombar)_9-(1,\omega,\ombar)_1&
(\ombar,\ombar,1)_9-(\omega,\ombar,1)_1&
(\omega,\omega,\ombar)_9-(1,\omega,\ombar)_1\cr
(0,0,1)_0&(\omega,\ombar,\omega)_9-(1,\ombar,\omega)_1&
(\omega,\ombar,\omega)_9-(1,\ombar,\omega)_1&
(\ombar,1,\ombar)_9-(\omega,1,\ombar)_1\cr\hline
\cr

&(1,0,0)_9&(0,1,0)_9&(0,0,1)_9\cr\hline
(1,0,0)_0 &(1,1,1)_1-(\ombar,1,1)_3&
(\omega,\omega,\omega)_1-(1,\omega,\omega)_3&
(\omega,\omega,\omega)_1-(1,\omega,\omega)_3\cr
(0,1,0)_0&(\ombar,1,\omega)_1-(\ombar,\omega,\ombar)_3&
(\ombar,1,\omega)_1-(\ombar,\omega,\ombar)_3&
(\omega,\ombar,1)_1-(\omega,1,\omega)_3\cr
(0,0,1)_0&(\ombar,\omega,1)_1-(\ombar,\ombar,\omega)_3&
(\omega,1,\ombar)_1-(\omega,\omega,1)_3&
(\ombar,\omega,1)_1-(\ombar,\ombar,\omega)_3\cr\hline
\cr

&(1,0,0)_3&(0,1,0)_3&(0,0,1)_3\cr\hline
(1,0,0)_1 &(\omega,1,1)_0-(\ombar,\ombar,\ombar)_9&
(1,\omega,1)_0-(\omega,1,\ombar)_9&(1,1,\omega)_0-(\omega,\ombar,1)_9\cr
(0,1,0)_1&(\omega,\omega,\ombar)_0-(1,\ombar,\omega)_9&
(1,\ombar,\ombar)_0-(\omega,\ombar,1)_9&
(1,\omega,1)_0-(1,1,1)_9\cr
(0,0,1)_1&(\omega,\ombar,\omega)_0-(1,\omega,\ombar)_9&
(1,1,\omega)_0-(1,1,1)_9&
(1,\ombar,\ombar)_0-(\omega,1,\ombar)_9\cr\hline
\cr

&(1,0,0)_9&(0,1,0)_9&(0,0,1)_9\cr\hline
(1,0,0)_3 &(\ombar,\ombar,\ombar)_0-(\ombar,1,1)_1&
(1,\omega,\ombar)_0-(\ombar,\omega,\ombar)_1&
(1,\ombar,\omega)_0-(\ombar,\ombar,\omega)_1\cr
(0,1,0)_3&(\omega,\omega,\omega)_0-(1,\ombar,1)_1&
(1,\omega,\ombar)_0-(\omega,\omega,1)_1&
(\omega,1,\ombar)_0-(\ombar,1,1)_1\cr
(0,0,1)_3&(\omega,\omega,\omega)_0-(1,1,\ombar)_1&
(\omega,\ombar,1)_0-(\ombar,1,1)_1&
(1,\ombar,\omega)_0-(\omega,1,\omega)_1\cr\hline
\cr

&(1,0,0)_1&(0,1,0)_1&(0,0,1)_1\cr\hline
(1,0,0)_9 &(1,\omega,\omega)_0-(\ombar,\omega,\omega)_3&
(1,1,\ombar)_0-(1,\omega,1)_3&(1,\ombar,1)_0-(1,1,\omega)_3\cr
(0,1,0)_9&(\ombar,\omega,\ombar)_0-(\ombar,1,\ombar)_3&
(\ombar,1,1)_0-(\ombar,\ombar,1)_3&
(\ombar,\ombar,\omega)_0-(\omega,1,1)_3\cr
(0,0,1)_9&(\ombar,\ombar,\omega)_0-(\ombar,\ombar,1)_3&
(\ombar,\omega,\ombar)_0-(\omega,1,1)_3&
(\ombar,1,1)_0-(\ombar,1,\ombar)_3\cr\hline
\end{array}$$
\caption{{\label{E6mult}}The Lie bracket on $E_6$}
\end{table}

\begin{lemma}\label{step4}
The product defined in Lemma~\ref{step3} is invariant under $G$.
\end{lemma}
\begin{proof}
It suffices now to prove that this product is 
invariant
under $c$. Since $c$ normalizes the group $3^{3+3}{:}3$ generated by $b$ together
with conjugates of $d$ and $e$, it suffices to check the product on one pair of
basis vectors in each orbit under the latter group. There are $26$ such orbits,
represented by 
\begin{eqnarray*}
&&[(1,0,0)_t,(1,0,0)_{t+1}],\cr
&&[(1,0,0)_t,(1,0,0)_{t+2}],
\end{eqnarray*} 
for each $t\in \FF_{13}$.

The easiest cases are those where $c$ acts monomially, that is on the
coordinates $t=0,1,3,4,9,X,T$. There are seven such cases, namely $T0$, $01$, $34$ and $9X$
of the form $t,t+1$, and $XT$, $T1$ and $13$ of the form $t,t+2$.
In the $01$ case we have
$$[(1,0,0)_0,(1,0,0)_1]=(\ombar,1,1)_3-(1,\ombar,\ombar)_9$$
and under $c$ the left-hand-side is fixed (since both factors are negated),
while on the right-hand-side the two terms are swapped. Thus this instance of
the product is preserved by $c$, as required. In the $T0$ case we have
$$[(1,0,0)_T,(1,0,0)_0]=(\ombar,1,1)_2-(1,\ombar,\ombar)_8,$$
and this time the left-hand-side is negated by $c$, while the right-hand-side maps to
$$(\omega,1,1)_2M_1-(1,\omega,\omega)_8M_2=-(\ombar,1,1)_2+(1,\ombar,\ombar)_8$$
as required, since $2+\ombar=\ombar\theta$ and $1+2\omega=\theta$.
Now consider the case $9X$: we have
$$[(1,0,0)_9,(1,0,0)_X]=(\ombar,1,1)_T-(1,\ombar,\ombar)_5$$
in which the left-hand-side maps under $c$ to $-[(\ombar,0,0)_3,(1,0,0)_4]$ and the
right-hand side maps to
$$(\omega,1,1)_T-(1,\omega,\omega)_6M_3=(\omega,1,1)_T-(\omega,\ombar,\ombar)_6$$
which checks out with Table~\ref{E6mult} after applying a suitable element of $D$.

In the case $13$ we have
$$[(1,0,0)_1,(1,0,0)_3]=(\omega,1,1)_0-(\ombar,\ombar,\ombar)_9$$
and the left-hand-side is mapped by $c$ to $[(1,0,0)_1,(\omega,0,0)_9]$, while
the right-hand-side is mapped to $-(\ombar,1,1)_0+(\omega,1,1)_3$,
which again checks out with Table~\ref{E6mult}.
Similarly the right-hand side of
$$[(1,0,0)_X,(1,0,0)_T]=(\omega,1,1)_9-(\ombar,\ombar,\ombar)_5$$
maps to 
$$-(1,\omega,\omega)_3-(\omega,\omega,\omega)_6M_3 =-(1,\omega,\omega)_3
+(\ombar,\omega,\omega)_6$$
which is the value of $[(1,0,0)_4,(1,0,0)_T]$ as required.
In the case $T1$ we have
$$[(1,0,0)_T,(1,0,0)_1]=(\omega,1,1)_E-(\ombar,\ombar,\ombar)_7,$$
and the second term on the right-hand-side is mapped by $c$ to
$$-(\omega,\omega,\omega)_EM_4=-(\omega,1,1)_E.$$ 
Since $c$ has order $2$
it also maps the first term on the right-hand-side to the negative of the second, and
so $c$ preserves this instance of the product also.

The other 19 of the 26 calculations are slightly more awkward since $c$
no longer acts monomially on the left-hand-side, and are left as exercises for the
reader. The calculations can be reduced from $19$ cases to $14$ by using the fact that
$c$ has order $2$.
These are the cases $$12/45/56/78/XE/02/57/79/E0$$ and one from each of the
pairs $$23/35, 24/46, 67/68, 89/9E, 8X/ET.$$
Indeed, one of these calculations was essentially done in Lemma~\ref{step2} above,
where the equivalence of the following was shown:
\begin{eqnarray*}
[(1,0,0)_8,(1,0,0)_9] &=& (\ombar,1,1)_E-(1,\ombar,\ombar)_4\cr
[(1,0,0)_8,(1,0,0)_3]&=& -(1,\omega,\omega)_7+(\ombar,\omega,\omega)_X
\end{eqnarray*}
Applying $b$ to the second equation we obtain 
$$[(\omega,0,0)_E,(\omega,0,0)_9]=-(\omega,\ombar,\ombar)_8+(1,\ombar,\ombar)_4$$
and thence
$$[(1,0,0)_9,(1,0,0)_E]=(\omega,\ombar,\ombar)_8-(\ombar,\omega,\omega)_4.$$
\end{proof}

To summarise the results of this section so far, we have now proved the following.
\begin{theorem}\label{producttheorem}
Up to scalar multiplication, there is a 
unique bilinear product on the $78$-space $V$ which is invariant under $G$.
This product is defined by
\begin{eqnarray*}
[(1,0,0)_0,(1,0,0)_1]&=&(\ombar,1,1)_3-(1,\ombar,\ombar)_9, 
\end{eqnarray*}
and 
 is anti-symmetric.
\end{theorem}

The only remaining serious calculation is to verify that our product
satisfies the Jacobi identity,
$$[[x,y],z]+[[y,z],x]+[[z,x],y]=0.$$
By linearity and anti-symmetry, and the symmetry of the formula under cyclically
permuting $x,y,z$,  it suffices to check this for unordered triples $\{x,y,z\}$ of
distinct basis vectors. 
\begin{proposition}\label{Jacobi}
The product defined in Lemma~\ref{step3} satisfies the Jacobi identity.
\end{proposition}
\begin{proof}
We need to check that
$$[[x_r,y_s],z_t] + [[y_s,z_t],x_r] + [[z_t,x_r],y_s]=0$$
for suitable choices of (linearly independent)
vectors $x_r,y_s,z_t$.
If $r,s,t$ are not collinear, then we may assume
\begin{eqnarray*}
x_r&=&(1,0,0)_0,\cr 
y_s&=&(1,0,0)_1,\cr 
z_t&=&(1,0,0)_2.
\end{eqnarray*}
If $r,s,t$ are collinear and distinct, then we may assume
\begin{eqnarray*}
x_r&=&(1,0,0)_0,\cr 
y_s&=&(1,0,0)_1,\mbox{ and}\cr 
z_t&=&(1,0,0)_3\mbox{ or }(\omega,0,0)_3. 
\end{eqnarray*}
If $r,s,t$ are collinear and two of them are equal,
we may assume that 
\begin{eqnarray*}
x_r&=&(1,0,0)_0,\cr 
y_s&=&(1,0,0)_1\mbox{ and }\cr
z_t&=&(\omega,0,0)_0\mbox{ or }(0,1,0)_0.
\end{eqnarray*} 

Thus we have five cases to check.
The four cases when $r,s,t$ are collinear are relatively easy, and left as
exercises. The hard case is when they are not collinear, and we calculate
as follows.
\begin{eqnarray*}
[[(1,0,0)_0,(1,0,0)_1],(1,0,0)_2]\cr + [[(1,0,0)_1,(1,0,0)_2],(1,0,0)_0]\cr
+ [[(1,0,0)_2,(1,0,0)_0],(1,0,0)_1] &=&
[(\ombar,1,1)_3,(1,0,0)_2] - [(1,\ombar,\ombar)_9,(1,0,0)_2]\cr
&&{}+
[(\ombar,1,1)_4,(1,0,0)_0] - [(1,\ombar,\ombar)_X,(1,0,0)_0]\cr
&&{}- [(\omega,1,1)_T,(1,0,0)_1] + [(\ombar,\ombar,\ombar)_8,(1,0,0)_1]
\end{eqnarray*}
Next we calculate the individual terms on the right-hand side
as follows (details of the 
calculations are omitted).
\begin{eqnarray*}
[(\ombar,1,1)_3,(1,0,0)_2]&=& -\theta(\omega,\ombar,\ombar)_5+\theta(\ombar,\omega,\omega)_E\cr
[(\ombar,1,1)_4,(1,0,0)_0]&=& 3(\ombar,0,0)_5+\theta(1,\ombar,\ombar)_7\cr
[(1,\ombar,\ombar)_X,(1,0,0)_0]&=&-\theta(\omega,\omega,\omega)_6 -3(\omega,0,0)_E\cr
[(1,0,0)_2,(1,\ombar,\ombar)_9]&=& \theta(\ombar,1,1)_6+3(\ombar,0,0)_7\cr
[(\omega,1,1)_T,(1,0,0)_1]&=& \theta(1,\omega,\omega)_E+\theta(\omega,\ombar,\ombar)_7\cr
[(\ombar,\ombar,\ombar)_8,(1,0,0)_1]&=& \theta(1,\ombar,\ombar)_5+\theta(1,\ombar,\ombar)_6
\end{eqnarray*}
Finally we combine these results and find that all the terms cancel out, giving
$0$ as required.
\end{proof}

\section{Identification of the algebra with $E_6$}
It is easy to see that the Lie product is not identically zero on any subspace
properly containing $V_0$, and therefore $V_0$ is a Cartan subalgebra.
The stabilizer of $V_0$ in $G$ is a group of shape $3^{3+3}{:}3^2{:}2S_4$
which acts on $V_0$ as $3^{1+2}{:}2S_4$. Since this group has no faithful 
complex representation
of degree less than $6$, it acts absolutely irreducibly on $V_0$.
Therefore the Lie algebra is simple, so by the classification theorem
it is $E_6$. Since $G$ acts irreducibly, the only invariant quadratic forms
are (positive or negative) definite. In particular the Killing form is negative
definite, so the algebra is the compact real form of $E_6$.

An alternative proof may be obtained from first principles
by extending the field to $\mathbb C$ and
explicitly diagonalizing the action of $V_0$ by multiplication on $L$,
and thereby obtaining the root spaces. These may then be explicitly identified
with the $72$ roots of the $E_6$ root system, and the products of the root
vectors explicitly calculated. One would then see that the complexification
of the algebra is the same
as the usual complex Lie algebra of type $E_6$.

As a first step, we compute the eigenspaces of the action of $V_0$. We find that
one of them is spanned by
$$v=(1,1,1)_1+(\omega,\ombar,\ombar)_3+(\omega,1,1)_9.$$
To verify this we compute
\begin{eqnarray*}
[(1,0,0)_0, v]&=&3(\ombar,1,1)_3 - 3(1,\ombar,\ombar)_9
+3(\ombar,\omega,\omega)_9-3(\ombar,\ombar,\ombar)_1\cr &&\qquad
+3(\omega,\omega,\omega)_1-3(1,\omega,\omega)_3\cr
&=& 3\theta v\cr
[(\omega,0,0)_0,v]&=& 3(1,\omega,\omega)_3-3(\ombar,\omega,\omega)_9
+ 3(1,\ombar,\ombar)_9-3(\omega,\omega,\omega)_1\cr
&&\qquad + 3(\ombar,\ombar,\ombar)_1-3(\ombar,1,1)_3\cr
&=& -3\theta v
\end{eqnarray*}
and
$$[(0,1,0)_0,v]=[(0,\omega,0)_0,v]=[(0,0,1)_0,v]=[(0,0,\omega)_0,v]=0.$$
Moreover, we see that this complex eigenspace corresponds to the pair of roots
$\pm(1-\omega,0,0)_0=\pm(\ombar\theta,0,0)_0$.
Applying suitable elements of $D$ we obtain the correspondence
\begin{eqnarray*}
\pm(\theta,0,0)&\leftrightarrow& \langle
(1,1,1)_1+(\ombar,1,1)_3+(1,\ombar,\ombar)_9,\cr
&&\qquad
(\omega,\omega,\omega)_1+(1,\omega,\omega)_3+(\omega,1,1)_9\rangle.
\end{eqnarray*}
Notice that `scalar multiplication' is always interpreted as applying a suitable
element of $D$, so is not always the same as scalar multiplication by $\omega$.
 The other orbits of $E$ give rise to the following correspondences:
\begin{eqnarray*}
\pm(1,1,1)_0&\leftrightarrow&\langle (\theta,0,0)_T+(\ombar,1,1)_2-(1,\ombar,\ombar)_8,\cr
&&\qquad(\omega\theta,0,0)_T+(1,\omega,\omega)_2-(\ombar,\omega,\omega)_8\rangle\cr
\pm(1,\omega,\omega)_0&\leftrightarrow&
\langle (\theta,0,0)_X-(\omega,\omega,\omega)_E+(\omega,1,1)_6,\cr
&&\qquad(\omega\theta,0,0)_X-(1,1,1)_E+(\ombar,\omega,\omega)_6\rangle\cr
\pm(1,\ombar,\ombar)_0&\leftrightarrow&
\langle (\theta,0,0)_4+(1,1,1)_5-(\ombar,1,1)_7,\cr
&&\qquad(\omega\theta,0,0)_4+(\omega,\omega,\omega)_5-(\omega,\ombar,\ombar)_7\rangle
\end{eqnarray*}

Now we can use elements of the stabiliser of $V_0$ in $3^{3+3}{:}\SL_3(3)$
to obtain all the other root spaces, labelled with the corresponding roots.
 The pointwise stabiliser of $V_0$ is an elementary abelian group
of order $3^5$, generated by 
$$d,d^a,e,c^{a^{-1}}cc^a,(c^{a^{-2}}c^{a^5}c^{a^{-2}}c)^2b.$$ 
It follows (or one can check directly)
that $v$ is an eigenvector for this group.
The elements $e^a, e^{a_2}$ then map the given eigenspace to the nine eigenspaces
which lie inside $V_1+V_3+V_9$. The eigenspaces lying in the other `lines' containing $0$ can
be computed by applying other conjugates of $c$ which fix the point $0$. 

Recall that our $39$-dimensional complex notation denotes a
$78$-dimensional real vector space.
In order to find a Chevalley basis, of course, one needs to
extend the scalars to $\mathbb C$ (without confusing the real vector $\omega$ with
the complex scalar $e^{2\pi i/3}$). Then each of our `eigenspaces' becomes
a $2$-dimensional space, in which one can distinguish two root vectors, corresponding
to a root and its negative.

\section{The subalgebra of type $F_4$}
The subspace $W_t$ of $V_t$ consisting of the vectors $(x,y,y)_t$ has (real)
dimension $4$. The direct sum $W$ of the $W_t$ is a space of dimension $52$, which is
easily seen to be invariant under the action of $a,b,c,d$. These elements in fact generate
a symmetry group of shape $3^3{:}\LL_3(3)$.

This group induces on each $W_t$ a group of shape $(3\times 2A_4){:}2$, which
acts irreducibly. Using the symmetry group it is not hard 
to show that $W$ is closed under the Lie product.
Hence $W$ is a simple Lie algebra of rank $4$ and dimension $52$ and can only be the
compact real form of $F_4$.

The short roots of the $F_4$ root system may be taken as the $24$ vectors
of the form
\begin{eqnarray*}
&&\pm \omega^n(\theta,0,0),\cr
&& \pm \omega^n(1,1,1),\cr 
&& \pm\omega^n(1,\omega,\omega),\cr
&& \pm\omega^n(1,\ombar,\ombar).
\end{eqnarray*} 
We may label these vectors by unit quaternions
by defining 
\begin{eqnarray*}
1&=&-(\theta,0,0),\cr 
i&=&(1,1,1),\cr 
j&=&(1,\omega,\omega),\cr
k&=&(1,\ombar,\ombar).
\end{eqnarray*}
and identifying (left-)multiplication by the complex number $\omega$ with left-multiplication
by the quaternion $\omega=(-1+i+j+k)/2$. Let $q_t$ denote the quaternion $q$
in the space $W_t$.

With this notation, 
the compact real form of the Lie algebra of type $F_4$ becomes a 
$13$-dimensional object over quaternions. It is of course not linear,
but the quaternions do provide a compact notation both for the multiplication and
for the action of certain automorphisms.
For example,
$d$ becomes left-quaternion multiplication by
$$(1,\omega,1,\omega,\omega,\omega,\ombar,\ombar,1,\omega,\ombar,\omega,1)$$
on the $13$ spaces $W_t$. Similarly, the element $b$ becomes $q_t\mapsto q_{3t}\omega$,
that is, the combination of 
{\em right}-quaternion-multiplication by $\omega$ with the 
coordinate permutation
$$(1,3,9)(2,6,5)(4,T,X)(8,E,7).$$

Similarly, the matrices $M_1$, $M_2$, $M_3$, $M_4$ defined earlier induce
right-multiplication by the quaternions $j\omega$, $i$, $k\omega$, $j$, respectively.
Complex conjugation induces the negative of the automorphism
$*$ which negates $i$ and
swaps $j$ with $-k$. Then $c$ maps 
\begin{eqnarray*}
q_0&\mapsto & q_0^*\cr
q_1&\mapsto & -q_1\cr
q_3&\leftrightarrow& -(q\omega)_9\cr
q_4&\leftrightarrow& -q_X^*\cr
q_T&\mapsto&-q_T^*\cr
q_2&\mapsto& -(q^*j\omega)_2\cr
q_8&\mapsto& -(q^*i)_8\cr
q_5&\leftrightarrow&-(q^*k\omega)_6\cr
q_7&\leftrightarrow&-(q^*j)_E
\end{eqnarray*}
For convenience, we note also
\begin{eqnarray*}
q_9&\mapsto& -(q\ombar)_3\cr
q_X&\mapsto& -q^*_4\cr
q_6&\mapsto& -(q^*i\omega)_5\cr
q_E&\mapsto& -(q^*k)_7
\end{eqnarray*}

The Lie product suitably scaled (in fact, this is the previous
product divided by $-3$) is written out in more
detail in Table~\ref{F4multtable}. 
All products of roots in the $W_t$ can be obtained from this table by applying
elements of the group $D\cong 3^3$.
Indeed, this table may be useful
for calculating the product in $E_6$ as well,
by applying elements of $E\cong 3^{3+3}$ to the entries.

\begin{remark}{\rm
One might expect that a result similar to Theorem~\ref{producttheorem}
should hold also for $F_4$. That is, one might conjecture that there is a unique
algebra invariant under the appropriate $52$-dimensional representation of
$3^3{:}\LL_3(3)$. However, this is not the case.}
\end{remark}

\begin{table}
$$\begin{array}{c|cccc|}
&\overline{\omega}1_1 & \overline{\omega}i_1 & \overline{\omega}j_1 & \overline{\omega}k_1\cr\hline
\ombar 1_0 & j_3+k_9 & \theta j_3+\theta k_9 & -j_3-k_9 & j_3+k_9\cr
\ombar i_0 & -i_3+j_9 & -j_3+k_9 & 1_3-1_9 & k_3-i_9\cr
\ombar j_0 & k_3+1_9 & -j_3+k_9 & i_3+i_9 & -1_3+j_9\cr
\ombar k_0 & -1_3-i_9 & -j_3+k_9 & -k_3-j_9 & -i_3+1_9\cr\hline
\cr&\overline{\omega}1_3 & \overline{\omega}i_3 & \overline{\omega}j_3 & \overline{\omega}k_3\cr\hline
1_0 & k_9+i_1 & k_9+i_1 & \theta k_9 + \theta i_1 & -k_9-i_1\cr
i_0 & -1_9-j_1 & -j_9+1_1 & -k_9+i_1&-i_9-k_1\cr
j_0 & -j_9+k_1 & i_9-j_1 & -k_9+i_1 & 1_9-1_1\cr
k_0 & i_9+1_1 & -1_9+k_1 & -k_9+i_1 & j_9+j_1\cr\hline
\cr&\overline{\omega}1_9 & \overline{\omega}i_9 & \overline{\omega}j_9 & \overline{\omega}k_9\cr\hline
{\omega} 1_0 & i_1+j_3&-i_1-j_3 & i_1+j_3 & \theta i_1+\theta j_3\cr{\omega} i_0 & j_1+1_3 & k_1+k_3 & -1_1+i_3 &-i_1+j_3\cr
{\omega} j_0 & -1_1-k_3&-j_1-i_3 & -k_1+1_3 & -i_1+j_3\cr
{\omega} k_0 & -k_1+i_3&1_1-1_3&j_1-k_3&-i_1+j_3\cr\hline
\cr&\ombar 1_3&\ombar i_3&\ombar j_3&\ombar k_3\cr\hline
 1_1 & -k_0-i_9&-i_0+j_9&1_0-k_9&j_0-1_9\cr
 i_1 & -1_0-k_9&-1_0-k_9&-\theta 1_0+\theta k_9 & 1_0+k_9\cr
 j_1 & i_0-1_9 & j_0+i_9 & -1_0+k_9 & -k_0+j_9\cr k_1 & -j_0+j_9&-k_0+1_9&1_0-k_9&i_0+i_9\cr\hline
\cr&1_9 & i_9 &j_9 & k_9\cr\hline
\ombar 1_3 & -i_0-j_1&k_0-1_1&-j_0+k_1&1_0-i_1\cr
\ombar i_3 & -k_0+k_1&j_0+j_1&-i_0+1_1&1_0-i_1\cr
\ombar j_3 & -1_0-i_1&1_0+i_1&-1_0-i_1&-\theta 1_0+\theta i_1\cr
\ombar k_3 & j_0-1_1 &-i_0+k_1&k_0+j_1&-1_0+i_1\cr\hline
\cr&{\omega}1_1 & {\omega}i_1 & {\omega}j_1 & {\omega}k_1\cr\hline
\omega 1_9 & -j_0-k_3&1_0-j_3&i_0-1_3&-k_0+i_3\cr
\omega i_9 & k_0-1_3&-1_0+j_3&-j_0+i_3&i_0+k_3\cr
\omega j_9 & -i_0+i_3&1_0-j_3&k_0+k_3&-j_0+1_3\cr
\omega k_9 & -1_0-j_3 &-\theta 1_0 +\theta j_3 &1_0+j_3&-1_0-j_3\cr\hline
\end{array}$$
\caption{\label{F4multtable}The multiplication table of the Lie algebra of type $F_4$}
\end{table}


In $3^3{:}\SL_3(3)$, the pointwise stabilier of $W_0$ is an elementary abelian group
of order $3^4$, generated by 
$$d,d^a,c^{a^{-1}}cc^a, (c^{a^{-2}}c^{a^5}c^{a^{-2}}c)^2b.$$
We may re-compute the action of $c^{a^{-1}}cc^a$ in $F_4$, as follows:
\begin{eqnarray*}
q_0&\mapsto& q_0,\cr
q_1&\mapsto& -(qk\ombar)_1,\cr
q_2&\mapsto& (qj\omega)_T,\cr
q_3&\mapsto& -(qi\ombar)_3,\cr
q_4&\mapsto& (q^*i\omega)_7,\cr
q_5&\mapsto& (qi\omega)_4,\cr
q_6&\mapsto& (qk\omega)_X,\cr
q_7&\mapsto& -(q^*j)_5,\cr
q_8&\mapsto& -(q^*i\ombar)_2,\cr
q_9&\mapsto& -(qj\ombar)_9,\cr
q_X&\mapsto& (q^*j\ombar)_E,\cr
q_E&\mapsto& -(q^*k\omega)_6,\cr
q_T&\mapsto& (q^*k)_8.
\end{eqnarray*}
The final element $(c^{a^{-2}}c^{a^5}c^{a^{-2}}c)^2b$ 
acts
as follows: 
\begin{eqnarray*}
q_0&\mapsto& q_0,\cr
q_1&\mapsto& (q\omega)_3,\cr
q_2&\mapsto& (qj)_T,\cr
q_3&\mapsto& -(qi)_9,\cr
q_4&\mapsto& (q\ombar)_4,\cr
q_5&\mapsto& -(qk\ombar)_5,\cr
q_6&\mapsto& (q^*\ombar)_E,\cr
q_7&\mapsto&(qk\omega)_7,\cr
q_8&\mapsto& (q^*k\omega)_2,\cr
q_9&\mapsto& (qi\ombar)_1,\cr
q_X&\mapsto&-(qj\omega)_6,\cr
q_E&\mapsto& (q^*i\omega)_X,\cr
q_T&\mapsto& (q^*j\omega)_8.
\end{eqnarray*}
We can find the common `eigenvectors' of this $3^4$
(where, again, scalar multiplication is defined by an element of $D$). 
These are the images under $D$ and $a$ of
\begin{eqnarray*}
\pm 1_0&\leftrightarrow&\langle i_1+\ombar j_3+ k_9, \omega i_1 + j_3+\omega k_9\rangle,\cr
\pm i_0&\leftrightarrow&\langle 1_T -\ombar j_2 + k_8, \omega_T -j_2+\ombar k_8\rangle,\cr
\pm j_0&\leftrightarrow&\langle 1_X+\omega i_E - \omega k_6, \omega_X+ i_E - \ombar k_6\rangle,\cr
\pm k_0&\leftrightarrow&\langle 1_4- i_5+\ombar j_7, \omega_4-\omega i_5+\omega j_7\rangle.
\end{eqnarray*}
The other eigenvectors can be found by applying elements of
the stabilizer of $W_0$.

More specifically, we can use the part of the Weyl group of $F_4$ that lies inside
$3^3{:}\SL_3(3)$. This is a group $(3 \times 2\udot A_4){:}2$, which is the
centralizer of an involution in $3^3{:}\SL_3(3)$. If we take the involution
$c^{a^{-1}}$, then it is centralized by $d^{a^{-1}}$ in the normal $3^3$, and by
$c^{a^{-2}}$ and $(bc^{a^{-1}}b)^2$, generating $2S_4$, inside $\SL_3(3)$.
For convenience we exhibit the element $(bc^{a^{-1}}b)^2$ explicitly:
\begin{eqnarray*}
&q_0\mapsto q\omega_0,&\cr
q_1\mapsto q_1,&q_3\mapsto qk\omega_3,&q_9\mapsto-qj\ombar_9,\cr
q_2\mapsto qj_X,&q_X\mapsto q\ombar_4,& q_4\mapsto -qi\omega_2,\cr
q_5\mapsto q^*j_8,&q_8\mapsto q^*\omega_6,& q_6\mapsto qi\ombar_5,\cr
q_7\mapsto q^*k_T,&q_T\mapsto q^*i_E,& q_E\mapsto -qk_7.
\end{eqnarray*}

If $\varepsilon$ is a scalar of order $3$, we may pick a root vector
$$e_1=i_1+\ombar j_3 + k_9 - \varepsilon(\omega i_1+j_3+\omega k_9),$$
where the subscript denotes the corresponding root in $W_0$. Then we can apply
elements of the above group $(3 \times 2\udot A_4){:}2$ to get the remaining
(long) root vectors. We have
\begin{eqnarray*}
e_i&=& -\omega_T+j_2-\ombar k_8 -\varepsilon(-\ombar_T+\omega j_2-\omega k_8),\cr
e_j&=&-\ombar_X-\ombar i_E+k_6-\varepsilon(-1X-\omega i_E+\omega k_6),\cr
e_k&=&-1_4+i_5-\ombar j_7 -\varepsilon(-\omega_4+\omega i_5-\omega j_7),
\end{eqnarray*}
and the corresponding negative roots are obtained by swapping the coefficients
$1$ and $-\varepsilon$ of the two halves of the vector.
Left-multiples by $\omega$ and $\ombar$ are easily obtained by applying
the element $d^{a^{-1}}$.

\section{Reducing modulo $p$}
If $p$ is any prime other than $2$ or $3$, then the entire calculation goes through
if we replace $\mathbb R$ by $\mathbb F_p$,
or indeed by any field $\mathbb F$ of characteristic $p$. 
In particular, Theorem~\ref{producttheorem}
and Proposition~\ref{Jacobi} hold in this more general setting. Throughout,
$\mathbb C$ is replaced by a $2$-dimensional space over $\mathbb F$, with
basis $\{1,\omega\}$. (One must be careful to distinguish between $\omega$ and an
element of order $3$ in $\mathbb F$, if there is one.)

In characteristic $3$, however, the whole strategy fails for many reasons: $G$ has no
faithful irreducible representations in characteristic $3$, a $2$-dimensional space does
not support a fixed-point-free linear map $\omega$ of order $3$, and the definition
of $c$ requires dividing  by $3$, to give just a few examples.

In characteristic $2$, 
the proof of Lemma~\ref{unique} fails, but the construction of the Lie algebra goes through.
Restricting to $F_4$ one finds that the group $3^3{:}\SL_3(3)$ is no longer irreducible
in characteristic $2$, but has two constituents, each of degree $26$. This is reflected
in the fact that
the Lie algebra of type $F_4$ is no longer simple, but
contains an ideal of dimension $26$. This ideal can be
seen using the ($2$-sided) ideal $\langle 1+i\rangle$ in $\mathbb Z[i,\omega]$.
The latter ideal is spanned additively by the long roots of the $F_4$ root system, and
modulo $2\mathbb Z[i,\omega]$ contains just three non-zero cosets, containing
respectively $i+j$, $j+k$ and $k+i$. Writing 
\begin{eqnarray*}
a_t&=&i_t+j_t,\cr
b_t&=&j_t+k_t,\cr 
c_t&=&k_t+i_t,
\end{eqnarray*}
 we have the following multiplication table for
the ideal:
$$\begin{array}{c|ccc|}
&a_1&b_1&c_1\cr\hline
a_0&c_3+c_9&a_3+a_9&b_3+b_9\cr
b_0&a_3+b_9&b_3+c_9&c_3+a_9\cr
c_0&b_3+a_9&c_3+b_9&a_3+c_9\cr\hline
\end{array}$$
Apart from the notation, this is the same as the multiplication constructed in
\cite{bigRee} for the exceptional Jordan algebra in characteristic $2$.
(This is the only characteristic in which the exceptional Jordan algebra
is also a Lie algebra.)

\end{document}